\documentclass[a4paper,12pt,leqno]{article}
\usepackage{latexsym}
\usepackage[all]{xy}

\usepackage{amssymb,amsxtra} 
\usepackage{amsmath} 
\usepackage{amsthm}
\usepackage{geometry}
\usepackage{enumerate}
\usepackage[colorlinks,backref]{hyperref}

\def\Z{{\mathbb{Z}}}
\def\Q{{\mathcal{Q}}}
\def\K{{\mathbb{K}}}
\def\A{{\mathcal{A}}}
\def\B{{\mathcal{B}}}

\newcommand{\Am}{(\mathcal{A},m)}
\newcommand{\Bm}{(\mathcal{B},m)}
\newcommand{\p}{\partial}
\newcommand{\xymat}{\SelectTips{cm}{}\xymatrix}

\DeclareMathOperator{\pdeg}{pdeg}
\DeclareMathOperator{\rank}{rank}
\DeclareMathOperator{\codim}{codim}

\DeclareMathOperator{\coker}{coker}
\DeclareMathOperator{\Der}{Der}

\numberwithin{equation}{section}

\theoremstyle{break}
\newtheorem{theorem}{Theorem}[section]
\newtheorem{prop}[theorem]{Proposition}
\newtheorem{cor}[theorem]{Corollary}
\newtheorem{lemma}[theorem]{Lemma}
\newtheorem{definition}[theorem]{Definition}
\newtheorem{rem}[theorem]{Remark}

\newtheorem{example}[theorem]{Example}

\title{Heavy hyperplanes in multiarrangements and their freeness}
\author{Takuro Abe\footnote{Institute of Mathematics for Industry, Kyushu University, Japan. Email: abe@imi.kyushu-u.ac.jp} and 
Lukas K\"{u}hne\footnote{Mathematical Institute, University of Bonn, Germany. Email: lf.kuehne@gmail.com}
}
\date{\today} 

\begin{document}

\maketitle

\begin{abstract}
Only few categories of free arrangements are known in which 
Terao's conjecture holds. One of such categories consists of $3$-arrangements 
with unbalanced Ziegler restrictions. In this paper, we generalize 
this result to arbitrary dimensional arrangements in terms of flags by 
introducing unbalanced multiarrangements.
For that purpose, we generalize several freeness criterions for simple arrangements, including
Yoshinaga's freeness criterion, to unbalanced multiarrangements.
\end{abstract}

\section{Introduction}
In the theory of hyperplane arrangements, the freeness of an arrangement is one of the most important 
objects to study. Terao's conjecture asserting the dependence of the 
freeness only on the combinatorics is the longstanding open problem in this area. The 
recent approach to this problem, which gives a partial answer, is based on 
multiarrangements due to Yoshinaga's criterion in \cite{Y}, \cite{Y2}, \cite{AY2} and 
\cite{A2}.

One of the most appealing results in this approach is the combinatorial dependence of
the freeness of a 3-arrangement with an unbalanced Ziegler restriction. Namely, let $\A$ be a 
$3$-arrangement and $(\A^H,m^H)$ the Ziegler restriction of $\A$ onto $H \in \A$, which is defined by
$\A^H :=\{H \cap H' \mid H' \in \A \setminus \{H\}\}$ and $m^H (X):= |\A_X\setminus \lbrace H \rbrace|$ for $X\in\A^H$.
We say that $(\A^H,m^H)$ is unbalanced if 
there is an $X \in \A^H$ such that $2m^H(X) \ge |m^H|$. 
For an unbalanced $2$-multiarrangement $(\A^H,m^H)$, the exponents are 
of the form $(m^H(X),|m^H|-m^H(X))$, which is determined by the combinatorics of 
$(\A^H,m^H)$ (and also $\A$). We say that $\A$ has an unbalanced Ziegler restriction if 
the Ziegler restriction of $\A$ onto one of $H \in \A$ is unbalanced. In 
conclusion, 
the freeness of such $\A$ depends only on $L(\A)$, i.e., 
such an $\A$ is free if and only if $\chi_0(\A;0)=m^H(X)(|m^H|-m^H(X))$ by \cite{Y}.

Hence it is a natural and interesting progression to question whether we can generalize this 
formulation to arrangements in an arbitrary dimensional vector space.
It was unclear in the beginning as to how we should define
an unbalanced multiarrangment in the vector space whose dimension is at least three. The first aim 
of this paper is to provide the definition of ``unbalanced'' multiarrangements in an arbitrary 
dimensional vector space. 
This definition is analogue to the one for 2-multiarrangements.

\begin{definition}
Let $(\A,m)$ be a multiarrangement in $V=\K^\ell$. 
\begin{itemize}
\item[(1)]
We say that a hyperplane $H_0\in\A $ is the \textbf{heavy hyperplane} with $m_0:=m(H_0)$ if
$2m_0 \ge |m|$. A multiarrangement $\Am$ is called \textbf{unbalanced} if it has the heavy hyperplane $H_0\in\A$.
\item[(2)]
We say that $(\A,m)$ has a \textbf{locally heavy hyperplane} $H_0 \in \A$ with 
$m_0:=m(H_0)$ if $2m_0 \ge |m_X|$ for all the localization 
$(\A_X,m_X)$ with $X \in \A^{H_0}$ satisfying 
$|\A_X| \ge 3$.
\item[(3)]
Let $H_0$ be a locally heavy hyperplane in $(\A,m)$. The 
\textbf{Euler-Ziegler restriction} $(\A^{H_0},m^{H_0})$ of $(\A,m)$ onto $H_0 \in \A$ is 
defined by, $\A^{H_0}:=\{H_0 \cap H \mid 
H \in \A \setminus \{H_0\}$, and 
$m^{H_0}(X):=
\sum_{H \in \A_X \setminus \{H_0\}} m(H)$.
\item[(4)]
Let $(\A,m)$ be an arbitrary multiarrangement and 
$H_0 \in \A$. The 
\textbf{multi-Ziegler restriction} $(\A^{H_0},m^{H_0})$ of $(\A,m)$ onto $H_0 \in \A$ is 
defined by, $\A^{H_0}:=\{H_0 \cap H \mid 
H \in \A \setminus \{H_0\}\}$, and 
$m^{H_0}(X):=\sum_{H \in \A_X \setminus \{H_0\}} m(H)$.
\end{itemize}
\label{heavy}
\end{definition}

Note that clearly the heavy hyperplane is also a locally heavy one. 
Definition \ref{heavy} (3) is a very simple generalization of the classical Ziegler restriction 
introduced in \cite{Z}. 
In general, this definition does not work as well as 
it does for simple arrangements. However, if we focus on 
a locally heavy hyperplane, then this simply generalized Ziegler multiplicity coincides with 
the Euler multiplicity defined in \cite{ATW}. Hence this restriction can be useful, as the following 
main result shows:

\begin{theorem}
\begin{itemize}
\item[(1)]
Let $(\A,m)$ be unbalanced with the heavy hyperplane $H_0 \in \A$ with 
$m_0:=m(H_0)$. Then 
$$
b_2(\A,m)-m_0(|m|-m_0) \ge b_2(\A^{H_0},m^{H_0}).
$$
Moreover, 
$(\A,m)$ is free 
if and only if the Euler-Ziegler restriction $(\A^{H_0},m^{H_0})$ of 
$(\A,m)$ onto $H_0$ is free, and 
$b_2(\A,m)-m_0(|m|-m_0)=b_2(\A^{H_0},m^{H_0})$. In this case, 
$\exp(\A,m)=(m_0,d_2,\ldots,d_\ell)$, where 
$\exp(\A^{H_0},m^{H_0})=(d_2,\ldots,d_\ell)$.
\item[(2)]
Let $(\A,m)$ have a locally heavy hyperplane $H_0 \in \A$ with 
$m_0:=m(H_0)$. Then 
$(\A,m)$ is free if the Euler-Ziegler restriction $(\A^{H_0},m^{H_0})$ of 
$(\A,m)$ onto $H_0$ is free, and 
$b_2(\A,m)-m_0(|m|-m_0)=b_2(\A^{H_0},m^{H_0})$. In this case, 
$\exp(\A,m)=(m_0,d_2,\ldots,d_\ell)$, where 
$\exp(\A^{H_0},m^{H_0})=(d_2,\ldots,d_\ell)$.
\end{itemize}
\label{AYheavy}
\end{theorem}

Theorem \ref{AYheavy} enables us to determine
the freeness of an arbitrary unbalanced multiarrangement. 
The determination of the freeness of a given multiarrangement is far more difficult
than that of a simple arrangement.
There are only two ways to validate it: Ziegler's free restriction theorem in \cite{Z},  
or the addition-deletion theorem in~\cite{ATW2}. To apply the former, 
we need to find a free simple arrangement whose Ziegler restriction is the given 
one. To apply the latter, we need a free multiarrangement close to the given one
and then increase/decrease multiplicities by applying the addition-deletion theorem in \cite{ATW2}. 
Theorem \ref{AYheavy} enables us
to check the freeness only by using the information of the given one under the heaviness assumption. Let us see 
how to apply it in the following example:

\begin{example}
Let $(\A,m)$ be a multiarrangement defined by 
$$
x^5y^2z^{16}(x-y)^3(y-z)^2(x-z)^4=0.
$$
This has the heavy hyperplane $H:=\{z=0\}$ by definition.
We can easily compute both sides in Theorem \ref{AYheavy} (1) as 
$$
b_2(\A,m)-m(H)(|m|-m(H))=63 =b_2(\A^H,m^H).
$$
Note that the Euler-Ziegler restriction of $(\A,m)$ onto $H$ 
is free with exponents $(7,9)$, see \cite{WY} for example. 
Hence $(\A,m)$ is free with exponents 
$(7,9,16)$ by Theorem \ref{AYheavy} (1). In fact, the 
multiarrangement $(\A,m+k\delta_H)$ is free when $
-7 \le k$ by Theorem \ref{AYheavy} (2).

Let $(\B,m)$ be a multiarrangement defined by 
$$
x^2y^2z^{14}(x-y)^3(y-z)^3(x-z)^4=0.
$$
This has the heavy hyperplane $H:=\{z=0\}$ by definition.
We can easily compute both sides in Theorem \ref{AYheavy} (1) as 
$$
b_2(\B,m)-m(H)(|m|-m(H))=51 >
49=b_2(\B^H,m^H).
$$
Hence $(\B,m)$ is not free.

\end{example}

In particular, as an easy corollary, we have the following.

\begin{cor}
Let $\A$ be the Weyl arrangement of the type $A_3$. Assume that 
$(\A,m)$ is unbalanced. Then the freeness of $\A$ depends only 
on $L(\A,m)$, where $L(\A,m)$ consists of $L(\A)$ with the information 
$m(H)$ for each $H \in L_1(\A)$.
\label{A3}
\end{cor}

Definition \ref{heavy} (4) is the simplest whose meaning has not yet been investigated.
This restriction also plays an interesting role for general multiarrangements as shown in the next 
definition and theorem.

\begin{definition}
Let $(\A,m)$ be a multiarrangment and fix $H \in \A$. Then define 
\begin{eqnarray*}
b_2^H(\A,m):&=&
b_2(\A,m)-m(H)(|m|-m(H)),\\
B_2^H(\A,m):&=&
\sum_{X \in L_2(\A) \setminus \A^H} b_2(\A_X,m_X).
\end{eqnarray*}
\label{reducedb2}
\end{definition}

\begin{theorem}\label{ineq2}
Let $(\A,m)$ be an $\ell$-multiarrangement and $H \in \A$. Then 
\[b_2^H(\A,m) \ge B^H_2(\A,m) \ge b_2(\A^H,m^H) \ge b_2(\A^H,m^*),\] 
where $m^*$ is the Euler multiplicity of $(\A,m)$ onto $H$ and $(\A^H,m^H)$
the multi-Ziegler restriction of $(\Am)$ onto $H$.
\end{theorem}

We will show the above results by using a special derivation, 
which plays a similar role as the Euler derivation for unbalanced multiarrangements. 
In other words, 
we can apply several arguments valid for simple arrangements in \cite{AY2}, \cite{Y}, 
\cite{Y2}, \cite{Z} and so on to unbalanced multiarrangements with slight modifications. In 
this paper we include all the proofs based on them for the completeness.

With Theorem~\ref{AYheavy} we can prove the following theorem.

\begin{theorem}\label{heavyflag}
Let $\A$ be an arrangement in $V=\K^\ell$ and
let $\{X_i\}_{i=0}^\ell$ be a flag of $\A$ such that 
$X_{i+1} \in \A^{X_i}$ is heavy in $(\A^{X_i},m^{X_i})$ for 
$i=0,\ldots,\ell-1$.
Then $\A$ is free with
$\exp(\A)=(m^{X_0}(X_{1}),m^{X_1}(X_{2}),\ldots,
m^{X_{\ell-1}}(X_{\ell}))$ if and only if 
\begin{equation}\label{flagEq}
b_2(\A)=\sum_{0 \le i< j \le \ell-1}m^{X_i}(X_{i+1}) m^{X_j}(X_{j+1})
\end{equation}
\end{theorem}

Let us call the flag in Theorem \ref{heavyflag} a 
\textbf{heavy flag}. Then Theorem \ref{heavyflag}
immediately implies the following corollary which shows 
Terao's conjecture in the class of arrangements that admit a heavy flag:

\begin{cor}
Let $HF_\ell$ be the set of hyperplane arrangements in 
$V=\K^\ell$ such that 
every $\A \in HF_\ell$ has a heavy flag. Then the freeness of 
$\A \in HF$ 
depends only on combinatorics. 
In this case, this flag gives both a
supersolvable filatration and a divisional flag as in \cite{A2}.
\label{combin}
\end{cor}

Note that Terao showed in~\cite{Ttalk} for a 3-arrangement $\A$ with a heavy flag that~\eqref{flagEq} is equivalent to $\A$ being supersolvable. Together with~\cite{Y} this implies Corollary~\ref{combin} for the case $\ell=3$. Corollary \ref{combin} generalizes it to the case of arbitrary dimension.

\begin{example}
Let $\A$ in $V=\K^4$ be defined by
$$
xy(x-y)(y-z)(x-z)(x+y+z)w(y-w)(y+w)\prod_{k=-3}^4(z-kw)=0,
$$
Then, $\A$ has the heavy flag $H:=\{w=0\} \in \A$, $L:=\{w=z=0\} \in (\A^H,m^H)$ and $K:=\{w=z=y=0\} \in (\A^L,m^L)$, with 
$m^V(H)=1$, $m^H(L)=8$, $m^L(K)=4$ and $m^K(0)=4$. 
So the RHS of~\eqref{flagEq} evaluates to $1(8+4+4)+8(4+4)+4\cdot4=96$. 
We can compute $b_2(\A)=99$.
Therefore, Theorem~\ref{heavyflag} implies that $\A$ is non-free.
\end{example}

\begin{rem}
Without the assumption on the heaviness, we can show the following:

for an arbitrary flag $\{X_i\}_{i=0}^\ell$ of $\A$.
$$
b_2(\A) \ge \sum_{0 \le i < j \le \ell-1}
m^{X_i}(X_{i+1})m^{X_j}(X_{j+1}).
$$
The equality holds if and only if $\A$ is supersolvable (hence free) with exponents
$(m^{X_0}(X_{1}),m^{X_1}(X_{2}),\ldots,
m^{X_{\ell-1}}(X_{\ell}))$.

For the proof, see \cite{A3}, Proposition 4.2 for example. However, in this case, 
$\A$ can be free even when the equality does not hold.  For example, the Weyl arrangement of the
type $D_\ell\ (\ell \ge 4)$ is free, but not supersolvable. Hence this gives an example of free arrangements which do not satisfy the $b_2$-equality above.
\end{rem}

The organization of this article is as follows. In \S2 we introduce several results used for the proof of results in \S1. In \S3 we prove Theorem~\ref{AYheavy}. 
In \S4 and \S5, we give several applications of our main result related to the freeness of simple arrangements and multiarrangements, and the combinatorial dependence of freeness.
Lemma~\ref{Le1}, Theorem~\ref{Th1} and Proposition~\ref{Prop1} are the main results of the second author's bachelor degree thesis~\cite{K}, completed at the University of Kaiserslautern under the supervision of Mathias Schulze. Their proofs are provided in this paper for the completeness.
\medskip

\textbf{Acknowledgements}.  
The first author is partially supported by JSPS Grant-in-Aid for Young 
Scientists (B) No. 24740012.
The main part of this work was done during a research stay of the second author at Kyoto University. He wishes to thank the German National Academic Foundation for funding this stay and Professor Mathias Schulze for the supervision of his Bachelor thesis at the University of Kaiserslautern.

\section{Preliminaries}

In this section we fix some notations and introduce some known results, which will be used later.

Let $V$ be a vector space of dimension $\ell$ over a field $\K$ and $S=S(V^*)$ be the symmetric algebra. We can choose coordinates ${x_1,\ldots,x_\ell}$ for $V^*$ such that $S=\K[x_1,\ldots,x_\ell]$. A hyperplane $H$ in $V$ is a linear subspace of codimension 1. A (central) \textbf{arrangement of hyperplanes} $\A$ is a finite collection of hyperplanes. This article is mainly concerned with \textbf{multiarrangements}, which are defined to be an arrangement of hyperplanes $\A$ with a multiplicity function $m : \A \rightarrow \Z_{>0}$. Multiarrangements were first defined by Ziegler in~\cite{Z} and are denoted by $\Am$.
Define $|m|=\sum_{H\in\A} m(H)$. A multiarrangement $\Am$ with $m(H)=1$ for all $H\in\A$ is sometimes called a simple hyperplane arrangement.
For each hyperplane $H$ we can choose a linear defining equation $\alpha_H \in S$. Then $\Q \Am=\prod_{H\in \A}\alpha_H^{m(H)}$ is the \textbf{defining polynomial} of a multiarrangement $\Am$.
For $L\in\A$ we define the \textbf{characteristic multiplicity} of $L$ to be $\delta_L(H)=1$ only if $H=L$ and 0 otherwise.

For an arrangement $\A$ the set of all non-empty intersections of elements of $\A$ is defined to be the \textbf{intersection lattice} $L(\A)$, i.e., 
$$
L(\A):=\{\cap_{H \in \B} H \mid \B \subset \A\}.
$$
It is ordered by reverse inclusion and ranked by the codimension. 
Denote by $L_r(\A)=\lbrace X\in L(\A) \mid \codim(X)=r \rbrace$ the set of $X\in L(\A)$ of codimension $r$.
For any $X\subset V$ let $\A_X = \lbrace  H \in \A \mid X\subseteq H\rbrace$ and $m_X=m|_{\A_X}$ be the \textbf{localization} of a multiarrangement $\Am$ at $X$.

For $p\geq 1$, the $S$-module $\Der^p (S)$ is the set of all alternating $p$-linear forms $\theta : S^p \to S$ 
such that $\theta$ is a $\K$-derivation in each variable. For $p=0$ one defines $\Der^0(S):=S$ and one also writes $\Der(S) :=\Der^1(S)$. Since $\Der (S)$ is a free $S$-module, 
it holds that $\Der^p (S)=\wedge^p \Der (S)$.
A derivation $\theta \in \Der^p (S)$ is called homogeneous of polynomial degree $p$ if $\theta(f_1,...,f_p)$ is zero or a polynomial of degree $p$ for all $f_1,...,f_p\in V^*$. It is denoted by $\pdeg \theta = p$.
The \textbf{logarithmic derivation modules} of $\Am$ are defined as
\[D^p \Am :=\lbrace \theta\in \Der^p(S) \mid \theta(\alpha_H,f_2,...,f_p) \in \alpha_H^{m(H)} S  \mbox{ for all }H\in \A \mbox{ and } f_2,...,f_p\in S\rbrace,\]
and $D\Am := D^1\Am$. If $D\Am$ is a free $S$-module, we call $\Am$ a free multiarrangement. 
The fact that the $S$-module $D\Am$ is reflexive implies the following proposition.

\begin{prop}\label{PrePro1}\cite[Theorem 5]{Z}
A multiarrangement $\Am$ is free whenever $r(\A) := \mbox{codim}_V(\bigcap_{H\in\A}H)\le 2$.

\end{prop}
If $\Am$ is free, one can choose a homogenous basis $\theta_1 , \ldots , \theta_\ell$ of $D\Am$ and define $\exp \Am =(\pdeg \theta_1 , \ldots, \pdeg \theta_\ell )$ to be the \textbf{exponents} of $\Am$.
For $\theta_1,...,\theta_\ell\in D\Am$ an $(\ell \times \ell)$-matrix $M(\theta_1,...,\theta_\ell)$ can be defined by setting the $(i,j)$-th entry to be $\theta_j(x_i)$. 
The determinant of this matrix gives a useful criterion to decide the freeness of a multiarrangement via the following theorem, first proved by K. Saito for simple arrangements~\cite[Theorem 1.8]{Sa} and by Ziegler for multiarrangements~\cite[Theorem 8]{Z}.

\begin{theorem}\label{PreTh2}
\textbf{(Saito's criterion)} Let  $\theta_1,...,\theta_\ell$ be derivations in $D\Am$. Then there exists some $f\in S$ such that
\[ \det M(\theta_1,...,\theta_\ell) = f \Q\Am.\] 
Furthermore, $\lbrace\theta_1,...,\theta_\ell\rbrace$ forms a basis of $D\Am$ if and only if $f \in \K^* $.
In particular, if $\theta_1,...,\theta_\ell$ are all homogeneous, then $\lbrace\theta_1,...,\theta_\ell\rbrace$ forms a basis of $D\Am$ if and only if the following two conditions are satisfied:
\begin{itemize}
\item[(i)] $\theta_1,...,\theta_\ell$ are independent over $S$.
\item[(ii)] $\sum_{i=1}^\ell \pdeg \theta_i =|m|$.
\end{itemize}  
\end{theorem}

One of the most important invariants of a (multi-) arrangement is its characteristic polynomial. It was first defined for simple arrangements combinatorially. We only use the following algebraic characterization, which is shown in~\cite{ST} for simple arrangements and in~\cite{ATW} for multiarrangements.
For a multiarrangement $\Am$, we define a function in $x$ and $t$.
\[\psi(\A,m;x,t):= \sum_{p=0}^\ell H(D^p(\A,m),x)(t(x-1)-1)^p ,\]
where $H(D^p(\A,m),x)$ is the Hilbert series of the graded $S$-module $D^p\Am$. 
Although $\psi(\Am;x,t)$ is a priori a rational function with poles along $x=1$, it is shown in~\cite[Theorem 2.5]{ATW} (see also \cite{ST}) that in fact $\psi(\A,m;x,t)$ is a polynomial in 
$x$ and $t$, and we can define the following.

\begin{definition}\cite[Definition 2,1]{ATW} 
A \textbf{characteristic polynomial} of a multiarrangement $\Am$ is defined as
\[\chi (\A,m;t) = \lim_{x\rightarrow 1 } (-1)^\ell \psi(\A,m;x,t). \]
Define the \textbf{Betti numbers} $b_i\Am \in \Z_{\ge 0}$ of $\Am$ by
\[\chi(\A,m;t) = t^\ell -b_1\Am t^{\ell-1} +b_2\Am t^{\ell-2}- \cdots + (-1)^\ell b_\ell \Am.\]
\end{definition}

Also in~\cite{ATW}, the following local-global formula is shown:

\begin{prop}\cite[Theorem 3.3]{ATW}\label{PrePro2}
Let $\Am$ be a multiarrangement. Then
\[b_k \Am= \sum_{X\in L_k(\A) }b_k(\A_X,m_X).\]
\end{prop}

Let $X\in L_2(\A)$. Since $(\A_X,m_X)$ is an arrangement of rank two, it is free. We denote its exponents by $(e_1(X),e_2(X),0,\ldots,0)$. 
With the local-global formula, we can express $b_1\Am$ and $b_2\Am$ as follows:

\begin{align}
b_1\Am =& |m| \\
b_2\Am =& \sum_{X\in L_2(\A) } e_1(X) e_2(X) \label{PreEq3}
\end{align}

If $\Am$ is free with exponents $(d_1,\ldots,d_\ell)$, Terao's factorization theorem for multiarrangements 
\cite[Theorem 4.1]{ATW} implies 
\begin{equation}\label{PreEq2}
\chi (\A,m;t) = \prod_{i=0}^\ell (t-d_i).
\end{equation}

Next, we review the addition-deletion theorem for multiarrangements. Let $\Am$ be a multiarrangement and $H_0\in\A$ a fixed hyperplane. 
The deletion $(\A',m')$ of $\Am$ with respect to $H_0$ is defined as follows:

\begin{enumerate}
\item[(1)] If $m(H_0)=1$, then $\A':=\A\setminus \lbrace H_0 \rbrace$ and $m'(H)=m(H)$ for all $H\in\A'$.
\item[(2)] If $m(H_0)\geq 2$, then $\A':=\A$ and for $H\in\A'=\A$, we set
\[m'(H)=
\left\{
	\begin{array}{ll}
		m(H)  & \mbox{if }  H \neq H_0,\\
		m(H) -1  & \mbox{if } H = H_0.
	\end{array}
\right.
\]
\end{enumerate}
The restricted arrangement is defined by $\A^{H_0}=\lbrace H_0 \cap H\mid H\in \A\setminus \lbrace H_0\rbrace \rbrace$. 
Now for $X\in \A^{H_0}$, $(\A_X,m_X)$ is a rank 2 arrangement and we can choose a basis $\lbrace \zeta_1,\ldots,\zeta_{\ell -2}, \theta_X, \psi_X \rbrace$ of $D(\A_X,m_X)$ such that $\pdeg \zeta_i =0$, $\theta_X \not\in \alpha_{H_0} \Der_\K (S)$ and $\theta_X \in \alpha_{H_0} \Der_\K (S)$ as shown in~\cite{ATW2}.
Then the \textbf{Euler multiplicity} $m^*$ 
on $\A^{H_0}$ is defined as $m^*(X) := \pdeg \theta_X$ and $(\A^{H_0},m^*)$ is called the \textbf{Euler restriction} of $\Am$. 
If $H_0$ is a locally heavy hyperplane in $\Am$, we can determine $m^*$ by~\cite[Proposition 4.1]{ATW2} as 
\begin{equation}\label{PreEq1}
m^*(X) = |m_X| - m(H_0)
\end{equation} 
for any $X\in \A^{H_0}$. Hence in the locally heavy case, the Euler restriction $(\A^{H_0},m^*)$ coincides with the Euler-Ziegler restriction $(\A^{H_0},m^{H_0})$ (this is the reason why we 
call our new restriction the Euler-Ziegler restriction). For the Euler-restriction one can prove the following theorem.

\begin{theorem}\cite{ATW2}\label{PreTh3}\textbf{(Addition-Deletion Theorem for multiarrangements)}
Let $\Am$ be a multiarrangment and $H_0\in\A$.
Then any two of the following three statements imply the third:
\begin{enumerate}[(1)]
\item\label{6.5a} $\Am$ is free with $\exp\Am=(d_1,...,d_{\ell-1},d_\ell)$.
\item\label{6.5b} $(\A',m')$ is free with $\exp(\A',m') =(d_1,...,d_{\ell-1},d_\ell-1)$.
\item\label{6.5c} $(\A^{H_0},m^*)$ is free with $\exp(\A^{H_0},m^*)=(d_1,...,d_{\ell-1})$.
\end{enumerate}
\end{theorem}

\begin{prop}\label{PrePro3}\cite[Proposition 1.6]{ATW}
Let $\Am$ be a multiarrangement and $X \in L(\A)$. If $\Am$ is free, then so is $(\A_X,m_X)$.
\end{prop}
Next we introduce supersolvable arrangements, which are defined as follows.

\begin{definition}\label{PreDef1}
An arrangement $\A$ is \textbf{supersolvable} if there exists a filtration 
\[ \A=\A_r \supset \A_{r-1} \supset \ldots \supset\A_2\supset\A_1\] such that
\begin{itemize}
\item[(1)] $\rank (\A_i ) = i$ for $i = 1, \ldots , r$.
\item[(2)] For any $H, H' \in \A_i \setminus \A_{i-1} $ with $H\neq H'$, there exists some $H'' \in \A_{i-1}$ such that $H \cap H' \subset H''$ .
\end{itemize}
\end{definition}

With Theorem~\ref{PreTh3} the following theorem is shown in~\cite{ATW2}.
\begin{theorem}\cite[Theorem 5.10]{ATW2}\label{PreTh1}
Let $\Am$ be a multiarrangement such that $\A$ is supersolvable with a filtration $\A=\A_r \supset \A_{r-1} \supset \ldots \supset\A_2\supset\A_1$ and $r\ge 2$.
Let $m_i$ denote the multiplicity $m \mid_{ \A_i }$ and $\exp(\A_2,m_2)=(d_1,d_2,0,\ldots,0)$.
Assume that for each $H' \in \A_d \setminus \A_{d -1 }$, $ H'' \in \A_{d -1 }$ for $d = 3, \ldots , r$ and $X := H' \cap H''$, it holds that
\[\A_X=\lbrace H',H'' \rbrace,\]
or that
\[m(H'')\ge \left( \sum_{ H\in (\A_d)_X \setminus (\A_{d - 1})_X } m(H)\right) -1.\]
Then $\Am$ is free with $\exp \Am = (d_1,d_2,|m_3| - |m_2|, \ldots , |m_r| - |m_{ r - 1}|,0 ,\ldots ,0)$.
\end{theorem}

The following theorems are fundamental in the study of freeness of simple arrangements. In the 
rest of this paper, we generalize these results to the setup of unbalanced multiarrangements.

\begin{theorem}\cite[Theorem 11]{Z}\label{PreTh4}
For a free arrangement $\A$ with $\exp (\A) = (1, d_2 , \ldots , d_\ell)$, the Ziegler restriction $(\A^{H_0},m^{H_0})$ is free for any hyperplane $H_0\in\A$ and $\exp (\A^{H_0}, m^{H_0} ) = ( d_2 , \ldots , d_\ell)$. 
\end{theorem} 

\begin{theorem}\cite[Theorem 5.1]{AY2}\label{PreTh5}
Let $\A$ be an arrangment with a fixed hyperplane $H_0\in \A$ and assume that $\ell \geq 3$. Let $(\A^{H_0},m^{H_0})$ be the Ziegler restriction onto $H_0$. Then it holds that 
\[b_2(\A) - (|\A|-1) \ge b_2 \left( \A^{H_0},m^{H_0} \right).\]
Moreover, 
$\A$ is free if and only if the above inequality is an equality, and 
$(\A^{H_0},m^{H_0})$ is free.
\end{theorem}

\section{Proof of Theorem~\ref{AYheavy}}

The Ziegler restrictions of 
free simple arrangements are also free by Theorem~\ref{PreTh4}. However, 
the multi-Ziegler restriction of a free multiarrangement is not necessarily free in general 
(e.g. see~\cite[Example 5.7]{K}). 
In the following, 
our first goal is to  generalize 
Theorem~\ref{PreTh4} to free unbalanced multiarrangements.

\begin{definition}
Fix a hyperplane $H_0\in \Am$ with $H_0=\ker \alpha_{H_0}$. Define the submodules $D_0^p\Am$ of $D^p\Am$ by
\[D_0^p\Am := \lbrace \delta \in D^p\Am \mid \delta (\alpha_{H_0},f_2,\ldots,f_p ) = 0 \mbox{ for any }f_i\in S \rbrace.\]
\end{definition}
The next Lemma was already shown in \cite[Theorem 11]{Z} for the Ziegler restriction of a simple arrangement and $p=1$. The proof transfers directly. We give it here for the completeness.
\begin{lemma}\label{Le1}
Let $\Am$ be a multiarrangement and fix a hyperplane $H_0\in \A$ with $H_0=\ker \alpha_{H_0}$.
If $\delta\in D_0^p\Am$, then $\delta\mid_{H_0} \in D^p(\A^{H_0},m^{H_0})$. In particular, this gives well-defined restriction maps $\mbox{res}_{H_0}^p : D_0^p \Am \rightarrow D^p (\A^{H_0},m^{H_0})$.
\end{lemma}

\begin{proof}
We may assume $\alpha_{H_0}=x_1$ by choosing suitable coordinates. Let $X\in\A^{H_0}$ and set
\[\A_X=\lbrace H_0,H_1,...,H_k\rbrace.\]
Hence $H_0 \cap H_{i}=X$ for all $i=1,...,k$.
We may assume that the defining linear forms $\alpha_i$ for $H_i$ are of the form
\[\alpha_i (x_1,...,x_\ell)=c_i x_1 + \alpha'(x_2,...,x_\ell)\]
for $i=1,...,k$, $\alpha' \in H_0^*$ a linear form and suitable $c_i\in \K$ which are pairwise distinct. In this situation, we have $\ker_{H_0} \alpha' = X$, i.e., $\alpha'$ is the defining equation of $X$ in $H_0$.
Let $\delta \in D_0^p \Am$ and $f_2,\ldots, f_p\in V^*$. By definition of the derivation module, 
it holds for $i=1,...,k$,
\[\delta (c_i x_1 + \alpha'(x_2,...,x_\ell),f_2,\ldots,f_p)\in (c_i x_1 + \alpha'(x_2,...,x_\ell))^{m(H_i)}S.\] 
Since $\delta(x_1,f_2,\ldots,f_p)=0$ due to the definition of $D_0^p\Am$, $(c_i x_1 + \alpha'(x_2,...,x_\ell))^{m(H_i)}$ is a divisor of $\delta ( \alpha'(x_2,...,x_\ell),f_2,\ldots,f_p)$ for all $i=1,...,k$. This fact yields
\[\delta ( \alpha'(x_2,...,x_\ell),f_2,\ldots,f_p)\in \prod_{i=1}^k (c_i x_1 + \alpha'(x_2,...,x_\ell))^{m(H_i)}S.\]
Now, the restricting onto $H_0$ gives us
\[\delta\mid_{H_0} ( \alpha',\overline{f_2},\ldots,\overline{f_p})\in (\alpha')^{\sum_{i=1}^k m(H_i)} \overline{S}=(\alpha')^{m^{H_0}(X)}\overline{S},\]
where $\overline{S}:=S/S x_1$ and $\overline{f_i}$ is the image of $f_i \in S$ under 
the canonical surjection $S \rightarrow \overline{S}$. Since $f_2,\ldots,f_p$ are chosen 
arbitrarily, it holds that $\delta|_{
H_0} \in D^p(\A^{H_0},m^{H_0})$.
\end{proof}

Next we introduce a particular type of derivation and show its existence for two special classes of multiarrangements. This derivation will play the role of the Euler derivation of a simple arrangement.   
\begin{definition}
Let $\Am$ be a multiarrangement with $H_0\in\A$. A derivation $\theta_0 \in D\Am$ is called a \textbf{good summand to $H_0$}, if $\pdeg \theta_0 =m(H_0)$ and $\theta_0 (\alpha_{H_0}) = \alpha_{H_0}^{m(H_0) }$. In this case, we also say $\Am$ has a good summand to $H_0$.
\end{definition}

Let $\theta_0$ be a good summand to $H_0$ of a multiarrangement $\Am$ with $H_0\in\A$.
Then we can decompose a general $\theta\in D\Am$ into 
\[\theta=\left( \theta-\frac{\theta(\alpha_{H_0})}{\alpha_{H_0}^{m(H_0)}}\theta_0 \right)+\frac{\theta(\alpha_{H_0})}{\alpha_{H_0}^{m(H_0)}}\theta_0.\]
Since $\theta-\frac{\theta(\alpha_{H_0})}{\alpha_{H_0}^{m(H_0)}}\theta_0 \in D_0\Am$, we obtain a decomposition
\[D\Am=S\theta_0 \oplus D_0\Am.\]
The next theorem shows that an unbalanced free multiarrangement has such a good summand. 

\begin{theorem}\label{Th1}
Suppose $\Am$ is a free multiarrangement with $\exp \Am =(d_1,\ldots,d_\ell)$ such that $H_0 \in \A$ is the heavy hyperplane with multiplicity $m_0=m(H_0)$. Then it holds that
\begin{itemize}
\item[(1)] \label{Th1.1} $\Am$ has a good summand to $H_0$, and 
$m_0 \in \exp(\A,m)$. Let $\exp(\A,m)=(m_0,d_2,\ldots,d_\ell)$.
\item[(2)] \label{Th1.2} The Euler-Ziegler restriction $(\A^{H_0},m^{H_0})$ of $\A$ 
onto $H_0$ is a free multiarrangement with exponents $\exp(\A^{H_0},m^{H_0})=(d_2,...,d_\ell)$.
\end{itemize}
\end{theorem}
\begin{proof}
We may choose coordinates such that $H_0=\{ x_1=0\}$. Let $\theta_1,...,\theta_\ell$ be a homogeneous basis of $D\Am$. By definition of $D\Am$ it holds for $i=1,...,\ell$
\begin{equation}\label{Eq1}
\theta_i (x_1) \in  x_1^{m_0}S.
\end{equation}
Assume $\pdeg \theta_i <m_0$ for all $i=1,...,\ell$. Thus~\eqref{Eq1} implies $\theta_i (x_1)=0$ for all $i=1,...,\ell$. 
Therefore,
\begin{equation}\label{Eq3}
\det M(\theta_1,...,\theta_\ell)=\det 
\begin{bmatrix}
0	&  \cdots	 & 0     \\ 
\theta_1(x_2)	&  \cdots	 & \theta_\ell(x_2)   \\
\vdots &  & \vdots \\
\theta_1(x_\ell)	&  \cdots	 & \theta_\ell(x_\ell)      
\end{bmatrix} =0.
\end{equation} 
However, with Theorem~\ref{PreTh2} we may assume 
\begin{equation}\label{Eq4}
\det M(\theta_1,...,\theta_\ell)= \Q\Am. 
\end{equation}  
In particular, $\det M(\theta_1,...,\theta_\ell)\neq 0$, which is a contradiction to~\eqref{Eq3}, and hence we may assume that $\theta_1(x_1) \neq 0$ with $\pdeg\theta_1 \geq m_0$.
Theorem~\ref{PreTh2} further yields
\begin{equation}\label{Eq5}
|m|=\sum_{i=1}^\ell \pdeg \theta_i.
\end{equation}
Together with our assumption $2m_0 \ge |m|$, this yields $\pdeg \theta_i <m_0$ for all $i=2,...,\ell$. Thus 
~\eqref{Eq1} shows that $\theta_i (x_1)=0$ for all $i=2,...,\ell$, and 
$\theta_1(x_1)=p  x_1^{m_0} $ for some $p\in S$. Hence 
\begin{equation}\label{Eq6}
\Q\Am=\det M(\theta_1,...,\theta_\ell)=p x_1^{m_0} \det
\begin{bmatrix}
 \theta_2(x_2)& \cdots	 & \theta_\ell(x_2)   \\
 \vdots & & \vdots \\
 \theta_2(x_\ell)	&  \cdots	 & \theta_\ell(x_\ell)      
\end{bmatrix}.
\end{equation}
Let $\theta_i':=\theta_i\mid_{x_1=0}$ be the restricted derivations for $i=2,...,\ell$. Due to Lemma~\ref{Le1}, $\theta_i' \in  D(\A^{H_0},m^{H_0})$ for all $i=2,...,\ell$. Therefore, we obtain from Theorem~\ref{PreTh2}
\begin{equation}\label{Eq7}
\det M(\theta_2',...,\theta_\ell') \in \Q(\A^{H_0},m^{H_0}) S.
\end{equation}
Since $x_1^{m_0}$ is the highest power of $x_1$ in $\Q\Am$,~\eqref{Eq6} implies
\begin{equation}\label{Eq8}
\det M(\theta_2',...,\theta_\ell') =\det 
\begin{bmatrix}
 \theta_2(x_2)& \cdots	 & \theta_\ell(x_2)   \\
 \vdots & & \vdots \\
 \theta_2(x_\ell)	&  \cdots	 & \theta_\ell(x_\ell)      
\end{bmatrix}\vert_{x_1=0}
\neq 0.
\end{equation}
This shows already $\pdeg \theta_i = \pdeg \theta_i'$ for $i=2,...,\ell$ and $\theta_2',\dots,\theta_\ell'$ are independent over $S/ (x_1)$. By~\eqref{Eq6} we have
\begin{equation}\label{Eq9}
\deg \det M(\theta_2',...,\theta_\ell') = \deg\Q\Am- \deg p x_1^{m_0} =|m|-m_0 -\deg p.
\end{equation}
On the other hand, we obtain by~\eqref{Eq7} together with~\eqref{Eq8}
\[ \deg \det M(\theta_2',...,\theta_\ell') \geq \deg \Q(\A^{H_0},m^{H_0})= |m| -m_0.\]
This inequality, together with~\eqref{Eq9}, yields $\deg p =0$, which shows that $\theta_1$ is a good summand to $H_0$. 
By~\eqref{Eq5} it holds that
\[|m^{H_0}|=|m|-m_0=\sum_{i=2}^\ell \deg \theta_i=\sum_{i=2}^\ell \deg \theta_i'.\]
So Theorem~\ref{PreTh2} immediately implies that 
$\{ \theta_2',...,\theta_\ell'\}$ forms a basis of $D(A^{H_0},m^{H_0})$.
\end{proof}

For general multiarrangements, we can show the following:

\begin{prop}\label{Prop2}
Let $\Am$ multiarrangement with $H_0\in\A$. Then $(\A,m+k\delta_{H_0})$ has a good summand to $H_0$ for all $k\gg 0$.
\end{prop}
\begin{proof}
Set $X_i:=\cap_{j=1}^i \{x_j=0\}$ ($i=1,\ldots,r$), where 
$x_1,\ldots,x_\ell$ form a $\K$-basis of $V^*$ and $r=\rank (\A)$.
We may assume that $H_0 = X_1$ and $H \supset X_r$ for all $H\in\A$.
By definition, for $H \in \A_{X_i} \setminus \A_{X_{i-1}}$, $\alpha_H$ can be chosen to be of the form $\alpha_H = a_1x_1+\ldots +a_{i-1}x_{i-1}+x_i$. 
Then $\A_i:=\A_{X_i}$ ($i=1,\ldots,r$) defines a filtration $\A=\A_r \supset \A_{r-1} \supset \ldots \supset\A_2\supset\A_1$ of $\A$. \\
We can obtain a supersolvable arrangement $\B$ with a filtration $\B=\B_r \supset \B_{r-1} \supset \ldots \supset\B_2\supset\B_1$ such that $\A_i \subset \B_i$ for $i=1,\ldots,r$ by adding hyperplanes to the filtered pieces $\A_i$ as follows:
First set $\B_r := \A_r$. Now inductively define $\B_{i}$ ($i=r-1,\ldots,2$) to be $\A_i$ with the additional hyperplanes $H'$
with $H'' \supset H \cap H'$ if $H'' \notin \B_{i}$, where $H,H'\in \B_{i+1} \setminus \A_{i}$ are distinct hyperplanes.
Finally, set $\B_1:=\A_1$ and the arrangement $\B$ is supersolvable by Definition~\ref{PreDef1}.\\
Now choose a multiplicity $m_\B$ on $\B$ such that $m_\B(H) \ge m(H)$ for any $H\in\A$, the hyperplane $H_0$ is heavy in $(\B,m_\B)$, and $(\B,m_\B)$ satisfies the assumptions of Theorem~\ref{PreTh1} as follows. First, we note that, for any distinct hyperplanes $H',H'' 
\in \A_r\setminus \A_{r-1}$, there exists the unique $L 
\in \A_{r-1}$ such that $L \supset H' \cap H''=:X$. Now replace 
$m(L)$ by $m(L) +k$ for sufficiently large $k$ in such a way that 
$\sum_{X \subset H \in \A_r \setminus \A_{r-1}} m(H) \le m(L)+k$. Proceed this 
replacement to obtain the desired new multiplicity $m_\B$.
Hence Theorem~\ref{PreTh1} implies that $(\B,m_\B)$ is free and Theorem~\ref{Th1} yields that the free and unbalanced multiarrangement $(\B,m_\B)$ has a good summand $\theta_0\in D(\B,m_\B)$ 
to $H_0$. 
Now choose $k\in\Z$ such that $m_\B(H_0)=(m+k\delta_{H_0})(H_0)$. Since 
$D(\A,m+k\delta_{H_0})\supset D(\B,m_\B)$ by definition of logarithmic vector fields, 
it holds that $\theta_0 \in D(\A,m+k\delta_{H_0})$, 
which completes the proof. 
\end{proof}

By~\eqref{PreEq1} the Euler multiplicity coincides with the Ziegler multiplicity for unbalanced multiarrangements.

\begin{prop}\label{Prop1}
Let $\Am$ be a multiarrangement with $H\in \A$ such that $H$ is the heavy hyperplane with multiplicity $m_0=m(H)$. Furthermore choose $k \in  \Z$ in such a way that $(\A,m+k\delta_H)$ is still unbalanced. Note that $k$ is allowed to be negative. Then $\Am$ is free if and only if $(\A,m+k\delta_H)$ is free.
\end{prop}
\begin{proof}
Assume that $\Am$ is free. 
Theorem~\ref{Th1} then implies that $\exp \Am = (m_0,d_2,\ldots d_\ell)$ holds and $(\A^{H_0},m^{H_0})$ is also free with $\exp (\A^{H_0},m^{H_0}) = (d_2,\ldots d_\ell)$. 
This Euler-Ziegler restriction 
$(\A^{H_0},m^{H_0})$ coincides with the Euler
restriction $(\A^{H_0},m^*)$ on $\A^{H_0}$ since $H_0$ is a heavy hyperplane. 
So repeated applications of Theorem~\ref{PreTh3} yields the freeness of $(\A,m+k\delta_H)$ since this arrangement is still heavy by assumption. Conversely, if one assumes that $(\A,m+k\delta_H)$ is free, the same argument shows the freeness of $\Am$.
\end{proof}

\begin{rem}\label{Rem2}
With the above notation, the Euler-Ziegler multiplicity $(\A^{H_0},m^{H_0})$ coincides with the Euler multiplicity $(\A^{H_0},m^*)$ on $\A^{H_0}$ by~\eqref{PreEq1} as long as $H_0$ is a locally heavy hyperplane. So, if $\Am$ is a free unbalanced multiarrangement, the above proof shows that $(\A,m+k\delta_H)$ is also a free multiarrangement for all $k \in \Z$ such that $H_0$ is a locally heavy hyperplane in $(\A,m+k\delta_H)$.
\end{rem}

\begin{theorem}\label{Th2}
Let $\Am$ be a multiarrangement with a fixed hyperplane $H_0 \in \A$. Assume that $\Am$ has a good summand $\theta_0\in D\Am$ with respect to $H_0$. Then
\[
(t-m_0) | \chi(\A,m;t).
\]
In this case define $\chi_0(\A,m;t):= \frac{1}{ t-m_0 }\chi(\A,m;t)$ to be the \textbf{reduced characteristic polynomial} of $\Am$.
\end{theorem}

The proof of Theorem~\ref{Th2} 
is similar to the one for simple arrangements in Proposition 4.4, 5.4 and 5.5 in~\cite{ST} by replacing the 
role played by the Euler derivation by the derivation with respect to the good summand. 
We give it here for the completeness.
\begin{proof}
Let $\alpha$ be the defining equation of the fixed hyperplane $H_0$. By definition of $D^p\Am$ we have $\theta(\alpha , f_2 , \ldots , f_p) \in \alpha^{m_0} S$, for $\theta \in D^p\Am$ and $f_2,\ldots,f_p\in S$. Thus we may define 
$\p : D^p \Am \rightarrow D^{p-1}\Am$ by
\[
\p(\theta) (f_2,\ldots , f_p) := \frac{ (i( \alpha )\theta)( f_2 , \ldots , f_p)}{\alpha^{m_0}}=\frac{\theta(\alpha , f_2 , \ldots , f_p)}{\alpha^{m_0}},
\]
where 
$i( \alpha ): \Der^p (S) \rightarrow \Der^{p-1} (S)$ is 
defined as $(i( \alpha ) \delta)(f_2,\ldots,f_p) := 
\delta (\alpha,f_2,\ldots,f_p)$ for $\delta \in \Der^p (S)$ and $f_i \in S$.
Thus, this map $\p$ is homogeneous of degree $-m_0$ with 
$\p(\phi \wedge \psi )=(\p\phi )\wedge \psi + (-1)^p \phi 
\wedge(\p\psi)$ for $\phi \in D^p\Am$ and $\psi \in D^q\Am$.
Since $\p \circ \p = 0$, this yields a complex
\begin{equation}\label{Eq10}
\xymat{
0 \ar[r] &D^\ell \Am \ar[r]^-{\p}& D^{\ell-1} \Am  \ar[r]^-{\p}&    \cdots  \ar[r]^-{\p}& D^0\Am \ar[r]& 0
}
\end{equation}
Now we use the good summand $\theta_0$ to show that this complex is acyclic. 
Let $\delta \in D^p\Am$ be a cycle, i.e., $\p \delta=0$. 
By the definition of $D^p\Am$ we know that $\theta_0 \wedge \delta \in D^{p+1}\Am$ which yields with the above skew-commutativity
\begin{align*}
\p(\theta_0 \wedge \delta) =& \p(\theta_0) \wedge \delta +(-1)^p \theta_0 \wedge \p(\delta) \\
=& \frac{\alpha^{m_0}}{\alpha^{m_0}} \delta + 0 = \delta.
\end{align*}
Hence, $\delta$ is a boundary and the complex~\eqref{Eq10} is acylic.\\
Let $H(D^p\Am,x)$ be the Hilbert series of the module $D^p\Am$. 
Since the complex~\eqref{Eq10} is acyclic and the boundary map 
$\p$ is of degree $-m_0$, 
it holds that 
\[
0 =  \sum_{p=0}^\ell (-1)^p H(D^p\Am,x)x^{m_0(\ell-p)}.
\]
Simplifying and dividing by $x^{m_0\ell}$ yields
\[
0 =  \sum_{p=0}^\ell  H(D^p\Am,x)\left(\frac{-1}{x^{m_0}}\right)^p.
\]
Using the definition
$\psi(\A,m;x,t):= \sum_{p=0}^\ell H(D^p(\A,m),x)(t(x-1)-1)^p $, we obtain:
\[
0= \psi \left( \A,m;x,\frac{\sum_{i=0}^{m_0-1} x^i}{x^{m_0}}\right)
\]
Since $\psi(\A,m;x,t)$ is a polynomial in $t$ and $x$ by \cite[Theorem 2.5]{ATW2},
we can substitute $x=1$ to obtain 
\[0 = (-1)^\ell \psi \left( \A,m;1,m_0\right)=\chi(\A,m;m_0)\]
by the definition of the characteristic polynomial $\chi (\A,m;t)=(-1)^\ell \psi (\A,m;1,t)$.
\end{proof}

In the situation of a multiarrangement with a good summand we can show the following.

\begin{lemma}\label{Le2}
Let $\Am$ be a multiarrangement with a good summand with respect to $H_0\in\A$ and $\chi_0(\A,m;t)$ the reduced characteristic polynomial. Then it holds that
\[\chi_0(\A,m;t) = t^{\ell-1} - (|m|-m_0)t^{\ell-1} +  b_2^{H_0} \Am t^{\ell-2} + \cdots +(-1)^\ell \chi_0(\A,m;0).\]
\end{lemma}
\begin{proof}
By definition, $\chi(\A,m;t)$ is of the form
\[
\chi(\A,m;t)=t^\ell -|m|t^{\ell-1}+b_2\Am t^{\ell-2} +\cdots +(-1)^{\ell} b_{\ell}(\A,m).
\]
Since $\chi(\A,m;t)=(t-m_0)\chi_0(\A,m;t) $, comparing coefficients completes the proof.
\end{proof}

The combination of the above results yields a result linking the freeness of an unbalanced multiarrangement with the freeness of its Euler-Ziegler restriction, 
which was first proved by Ziegler in~\cite{Z} for simple arrangements and its Ziegler restriction.

\begin{theorem}\label{Th3}
Let $\Am$ be a multiarrangement with the heavy hyperplane $H_0\in\A$. Then the following are equivalent:
\begin{itemize}
\item[(i)] $\Am$ is free.
\item[(ii)] The Euler-Ziegler restriction $(\A^{H_0},m^{H_0})$ is free and the restriction map $\mbox{res}_{H_0}^1 : D_0 \Am \rightarrow D (\A^{H_0},m^{H_0})$ defined in Lemma~\ref{Le1} is surjective.
\end{itemize}
\end{theorem}
\begin{proof}
Proposition~\ref{Prop2} enables us to choose $k>0$ such that $(\A,m+k\delta_{H_0})$ has 
a good summand to $H_0$. 
By Proposition~\ref{Prop1} $\Am$ is free if and only if $(\A,m+k\delta_{H_0})$ is free. 
Since $(\A^{H_0},m^{H_0})=(\A^{H_0},(m+k\delta_{H_0})^{H_0})$, $D_0 \Am =D_0 (\A,m+k\delta_{H_0}) $ and the restricition maps coincide, we may 
assume that $\Am$ itself has a good summand $\theta_0 \in D\Am$ to $H_0$.

The implication $(i)\Rightarrow (ii)$ was shown in Theorem~\ref{Th1}. 
For the converse, assume that the map $\mbox{res}_{H_0}^1$ is surjective, $(\A^{H_0},m^{H_0})$ is free and $\theta_1,\ldots \theta_{\ell-1}$ form a basis of $ D (\A^{H_0},m^{H_0})$. By the surjectivity of $\mbox{res}_{H_0}^1$, we can find derivations $\theta'_i \in D_0\Am$ for $i=1,\ldots,\ell-1$ such that $\mbox{res}_{H_0}^1 (\theta'_i)=\theta_i$. Then by Saito's criterion~\ref{PreTh2}, 
the derivations 
$\theta_0,\theta'_1,\ldots,\theta'_{\ell-1}$ form a basis of $D\Am$. Hence $\Am$ is free.
\end{proof}

Now we can give generalizations of Theorem 2.2 in \cite{Y2} and 
Theorem 3.2 in \cite{Y} in the following. 

\begin{prop}\label{Prop3}
Let $\Am$ be a multiarrangement in $V=\K^\ell$ with a fixed heavy hyperplane $H_0$ and $\ell \geq 4$. Assume that 
$\A$ is essential. Then $\Am$ is free if and only if $(\A^{H_0},m^{H_0})$ is free and $\Am$ is locally free along $H_0$, i.e., $(\A_X,m_X)$ is free for all $X\in L(\A)$ with $X\subset H_0$ and $X \neq \lbrace 0 \rbrace$. 
\end{prop}
\begin{proof}
As in the above proof, we may 
assume that $\Am$ has a good summand $\theta_0 \in D\Am$ to $H_0$ by increasing the multiplicity on $H_0$. 
Thus, there is a decomposition $D\Am = S\theta_0  
\oplus D_0\Am$. Then the proof is the same as in ~\cite[Theorem 2.2]{Y2} which shows the equivalence of the local freeness and the surjectivity of the residue map. 
Hence the result follows directly from Theorem~\ref{Th3}. Here we give another proof 
by using the result in \cite{AY}.

Since the ``only if'' part is easy, we show the ``if'' part. 
Since the $S$-module $D\Am$ is reflexive, so is its direct summand $D_0\Am$. 
From now on, consider the reflexive $\mathcal{O}_{\mathbb{P}^{\ell-1}}$-module
$\widetilde{D_0\Am}$, which is the sheafification of the graded $\K[x_1,\ldots,x_\ell]$-module 
$D_0(\A,m)$. 
It is known by~\cite{MS} that the stalk $\widetilde{D_0\Am}_{\bar{x}}$ for 
$\bar{x}\in \mathbb{P}^{\ell-1}$ is isomorphic to $\widetilde{D_0 (\A_x,m_x)}_{\bar{x}}$ for $x\in V\setminus\{ 0\}$,
where $\bar{x}\in \mathbb{P}^{\ell-1}$ is the point determined by $x\in V\setminus\{ 0\}$.
For $x\in H_0\setminus\{0\}$, the multiarrangement $(\A_x,m_x)$ is 
unbalanced and free by the assumptions. Hence 
Theorem~\ref{Th3} shows that the map $\mbox{res}_{H_0}^1$ is locally surjective, i.e., surjective as a sheaf homomorphism. Thus we obtain an exact sequence of sheaves on $\mathbb{P}^{\ell-1}$:
\begin{equation}\label{ExSheaf}
\xymat{
0 \ar[r] &\widetilde{D_0\Am}(-1)  \ar[r]^-{\cdot \alpha_{H_0}} &\widetilde{D_0\Am} \ar[r]^-{\widetilde{\mbox{res}}_{H_0}^1}&   \widetilde{D\left( \A^{H_0},m^{H_0}\right)} \ar[r]& 0
}
\end{equation}
$(\A^{H_0},m^{H_0})$ is free by the assumption, 
$\widetilde{D\left( \A^{H_0},m^{H_0}\right)}$ splits into the direct sum of line bundles. 
Also, clearly 
$$
\mbox{coker}(
\widetilde{D_0\Am}(-1) \stackrel{\cdot \alpha_{H_0}}{\rightarrow}  \widetilde{D_0\Am} 
) \simeq 
\widetilde{D_0\Am}|_{H_0}.
$$
Hence $\widetilde{D\left( \A^{H_0},m^{H_0}\right)} \simeq \widetilde{D_0\Am}|_{H_0}$ 
also splits into the direct sum of line bundles. Then apply the splitting criterion of reflexive sheaves (Theorem 0.2,  
\cite{AY}) to confirm that $\widetilde{D_0\Am}$ splits, equivalently $(\A,m)$ is free.
\end{proof}

Recall the definitions of $b_2^H(\A,m)$ and $B_2^H(\A,m)$ in Definition 
\ref{reducedb2}. Let us investigate some properties of these second Betti numbers.

\begin{prop}\label{Prop4}
Let $\Am$ be a multiarrangement in $V=\K^3$ with a fixed heavy hyperplane $H_0$. Then it holds that
\begin{equation}\label{Eq11}
b_2^{H_0} \Am \geq b_2 (\A^{H_0},m^{H_0}).
\end{equation}
Furthermore, $\Am$ is free if and only if~\eqref{Eq11} is an equality.
\end{prop}

\begin{proof}
Note that 
we can prove Proposition \ref{Prop4} by using the method in the proof of Theorem 3.2 in 
\cite{Y} with a slight modification. 
Here we give the proof for the completeness. 
$b_2^{H_0} \Am$ and $ b_2 (\A^{H_0},m^{H_0})$ are independent of the heavy multiplicity on $H_0$. Hence 
to determine freeness, Proposition~\ref{Prop1} enables us to assume that 
$\Am$ has a good summand to $H_0$ by increasing the multiplicity of $H_0$.
By Lemma~\ref{Le1} there are well-defined restriction maps $\mathrm{res}_{H_0}^p: D_0^p \Am \rightarrow D^p(\A^{H_0},m^{H_0})$ for $p=0,1,2$. 
Denote by $M^p \subset D^p(\A^{H_0},m^{H_0}) $ the image of the maps $ \mathrm{res}_{H_0}^p$.
Our first aim is to deduce a relation between the Hilbert series of $M^p$ and the one of $D^p\Am$ via the modules $D_0^p\Am$.
Since $\Am$ has a good summand to $H_0$, the complex~\eqref{Eq10} is again acyclic and 
we have $\ker \p = D_0^p\Am$. 
This yields short exact sequences for $p=0,1,2,3$:
\[\xymat{
0 \ar[r] &D_0^p \Am \ar[r]& D^p \Am   \ar[r]^-{\p}&    D_0^{p-1}\Am \ar[r]& 0,
}\]
where the first map is the inclusion. 
Since the map $\p$ is of degree $-m_0$, we obtain the following for the Hilbert series:
\[ H(D^p\Am,x) = H(D_0^p\Am,x) + x^{m_0} H(D_0^{p-1}\Am,x).\]
Multiplying each equation by $y^p$ and summing yields
\begin{equation}\label{Eq12}
\sum_{p=0}^3 H(D^p\Am,x)y^p = (1+yx^{m_0}) \sum_{p=0}^2 H(D_0^{p}\Am,x)y^p.
\end{equation}
Note that 
$D_0^3\Am =0$ by the injectivity of the map $\p : D^3\Am \rightarrow D^2\Am$.
On the other hand, we also have the following exact sequences for $p=0,1,2$:
\[\xymat{
0 \ar[r] & \alpha_{H_0} D_0^p \Am \ar[r]& D_0^p \Am   \ar[r]^-{ \mbox{res}_{H_0}^p}&    M^p\ar[r]& 0,
}\]
where the first map is the inclusion. 
Since these maps are all of degree $0$, 
\[ H(D_0^p\Am,x) = H(\alpha_{H_0} D_0^p\Am,x) +H(M^p,x). \]
Thus we obtain for $p=0,1, 2$
\begin{equation}\label{Eq13}
H(M^p,x)= (1-x) H( D_0^p\Am,x).
\end{equation}
Combining Equations~\eqref{Eq12} and~\eqref{Eq13} gives us the desired relation
\begin{equation}\label{Eq14}
\sum_{p=0}^3 H(D^p\Am,x)y^p = \frac{1+yx^{m_0}}{1-x} \sum_{p=0}^2 H(M^p,x) y^p.
\end{equation}
Hence we get by the definition of the reduced characteristic polynomial
\begin{align}
\nonumber \chi_0 (\A,m;t)  =& \frac{1}{t-m_0} \chi (\A,m;t) \\
\nonumber =& \frac{(-1)^3}{t-m_0} \lim_{x\rightarrow 1} \sum_{p=0}^3 H(D^p\Am,x)(t(x-1)-1)^p  \\
\nonumber \stackrel{\eqref{Eq14} }{=} &\frac{-1}{t-m_0} \lim_{x\rightarrow 1} \frac{1+x^{m_0}(t(x-1)-1)}{1-x} \sum_{p=0}^2 H(M^p,x)(t(x-1)-1)^p  \\
=&\lim_{x\rightarrow 1} \sum_{p=0}^2 H(M^p,x)(t(x-1)-1)^p \label{Eq15}.
\end{align}
Since 
$(\A^{H_0},m^{H_0})$ is a multiarrangement of rank two, it is free. 
Let $\exp \Am=(d_2,d_3)$. Hence the equation~\eqref{PreEq2} implies 
that 
\[\chi (\A^{H_0},m^{H_0};t) =  \lim_{x\rightarrow 1} \sum_{p=0}^2 H(D^p(\A^{H_0},m^{H_0}) ,x)(t(x-1)-1)^p = (t-d_2)(t-d_3).\]
Since $d_2+d_3=|m|-m_0$, 
Lemma~\ref{Le2} implies that 
$\chi_0 (\A,m;0) - d_2d_3 = \chi_0 (\A,m;t) - (t-d_2)(t-d_3)$. 
Since the maps $\mbox{res}_{H_0}^2$ and $\mbox{res}_{H_0}^0$ are surjective by 
definition, this combined with~\eqref{Eq15} implies
\begin{align*}
\chi_0 (\A,m;0) - d_2d_3 =& \chi_0 (\A,m;t) - (t-d_2)(t-d_3) \\
=& \lim_{x\rightarrow 1} (t(x-1)-1) \left[ H(M^1,x)- H(D^1(\A^{H_0},m^{H_0}) ,x) \right]\\
=& t \lim_{x\rightarrow 1}(x-1) \left[H(M^1,x)- H(D^1(\A^{H_0},m^{H_0}) ,x)\right] \\
& -\lim_{x\rightarrow 1} \left[H(M^1,x)- H(D^1(\A^{H_0},m^{H_0}) ,x)\right]
\end{align*}
Since the left hand side is a constant, we obtain that
\[ \chi_0 (\A,m;0) - d_2d_3 =\lim_{x\rightarrow 1} \left[ H(D^1(\A^{H_0},m^{H_0}) ,x)-H(M^1,x)\right]
\ge 0 \]
is finite and equal to $\dim \coker \mbox{res}_{H_0}^1 $. 
Since $ b_2 (\A^{H_0},m^{H_0}) = d_2d_3 $ and $b_2^{H_0} \Am = \chi_0 (\A,m;0)$ by Lemma~\ref{Le2}, the 
inequality~\eqref{Eq11} is shown. This argument also shows that inequality~\eqref{Eq11} is in fact an equality if and only if the map $\mbox{res}_{H_0}^1$ is surjective. 
Hence Theorem~\ref{Th3} completes the proof.
\end{proof}

By Propositions \ref{Prop3} and \ref{Prop4}, we can prove 
Theorem \ref{AYheavy} in the following by slightly modifying an argument given in \cite{AY2}:

\begin{proof}[Proof of Theorem~\ref{AYheavy}]
(1)\,\,
Recall Proposition~\ref{PrePro2}:
\[ b_2 \Am = \sum_{X\in L_2(\A)} b_2 (\A_X,m_X).\]
Since 
$b_2^{H_0} \Am=B_2^{H_0} \Am$ for the locally heavy multiarrangement $\Am$, 
we have 
\[ b_2^{H_0} \Am =b_2 \Am -m_0(|m|-m_0)= \sum_{X\in L_2(\A), X\not\subset H_0} b_2 (\A_X,m_X).\]
For $Y\in L_2(\A^{H_0})$, 
\[b_2^{H_0}  (\A_Y,m_Y) = \sum_{X\in L_2(\A), X\not\subset H_0, Y=X\cap H_0} b_2 (\A_X,m_X).
\]
Hence 
\[
b_2^{H_0}  (\A,m)=\sum_{X\in L_2(\A), X\not\subset H_0} b_2 (\A_X,m_X)=\sum_{Y\in L_2(\A^{H_0})} b_2^{H_0} (\A_Y,m_Y).
\]

Since $(\A_Y,m_Y)$ is an unbalanced multiarrangement of rank 3 with 
the heavy hyperplane $H_0$ 
for $Y\in L_2(\A^{H_0})$, Proposition~\ref{Prop4} shows that 
\begin{align}
b_2^{H_0} \Am &=  \sum_{Y\in L_2(\A^{H_0})} b_2^{H_0} (\A_Y,m_Y)\nonumber \\
 &\geq  \sum_{Y\in L_2(\A^{H_0})} b_2 (\A_Y^{H_0},m_Y^{H_0})=b_2 (\A^{H_0},m^{H_0})\label{Eq17}.
\end{align} 
Now assume that $(\A^{H_0},m^{H_0})$ is free. Then we claim that the following three statments are equivalent:
\begin{itemize}
\item[(i)] $\Am$ is free.
\item[(ii)] $b_2^{H_0}\Am = b_2 (\A^{H_0},m^{H_0})$.
\item[(iii)] $\Am$ is locally free in codimension 3 along $H_0$, i.e., 
$(\A_X,m_X)$ is free for all $X \in L_3(\A)$ with $X \subset H_0$.
\end{itemize}
Note that $(i)\Rightarrow(iii) $ follows immediately from Proposition~\ref{PrePro3}. 
Also, Proposition~\ref{Prop4} applied to each summand of~\eqref{Eq17} shows that
$(ii)$ and $(iii)$ are equivalent.

We will prove $(iii) \Rightarrow (i)$ by induction on $\ell$. 
If $\ell =3$, then $L_3(\A)=\lbrace 0 \rbrace$ and $(\A_X,m_X)=\Am$ for $X\in L_3(\A)$. Hence  $\Am$ is free. Next, let $\ell \geq 4$. We will show that $\Am$ is locally free along $H_0$, which implies that $\Am$ is free by Proposition~\ref{Prop3}. Let $Z \in L(\A)$ with $0 \neq Z \subset H_0$. Then $(\A_Z,m_Z)$ has rank at most $\ell - 1$.
Since $Z\subset H_0$, $(\A_Z,m_Z)$ is still unbalanced and since $(\A_Z^{H_0},m_Z^{H_0})$ is a localization of $(\A^{H_0},m^{H_0})$, it is free by Proposition~\ref{PrePro3}. 
Finally, 
we show that $(\A_Z,m_Z)$ satisfies (iii) and hence by induction assumption, $(\A_Z,m_Z)$ is free. 
So let $W\in L_3(\A_Z)$ with $W \subset H_0$. Then 
$W\in L_3 (\A)$, and 
by definition of the localization, 
\begin{equation}\label{Eq16}
 ((\A_Z)_W,(m_Z)_W)=(\A_W,m_W).
\end{equation}
By assumption, $(\A_W,m_W)$ is free. 
Hence $((\A_Z)_W,(m_Z)_W)$ is free too by 
~\eqref{Eq16}. 
Therefore, $(\A_Z,m_Z)$ satisfies (iii) and the proof of (1) is completed.

For (2) choose $k\in\Z_{\ge 0}$ such that $H_0$ is a heavy hyperplane in $(\A,m+k\delta_{H_0})$. Then it clearly holds that $(\A^{H_0},m^{H_0})=(\A^{H_0},(m+k\delta_{H_0})^{H_0})$ and $b_2^{H_0}\Am=b_2^{H_0}(\A,m+k\delta_{H_0})$. Hence by (1), $(\A,m+k\delta_{H_0})$ is free. Finally, Remark~\ref{Rem2} also implies that $\Am$ is free, since $H_0$ is locally heavy in $\Am$.
\end{proof}

\section{Heavy flags}

In this section we will prove Theorem~\ref{ineq2} and Theorem~\ref{heavyflag}. First, we quote the following lemma.

\begin{lemma}\label{LemIneq3}\cite[Lemma 4.3]{AN}
Let $\Am$ be a multiarrangement with $\exp\Am=(d_1,d_2)$. Then for $H\in\A$ it holds that $\exp (\A,m-\delta_H)=(d_1 -1,d_2)$ or $\exp (\A,m-\delta_H)=(d_1,d_2 -1)$. In particular, 
for two multiplicities $m,m' : \A \rightarrow \Z_{>0}$ with 
$m(H) \ge m'(H)$ for all $H \in \A$, it holds that 
$b_2(\A,m) \ge b_2(\A,m')$.
\end{lemma}

Second, we show the following lemma.

\begin{lemma}\label{LemIneq2}
Let $\Am$ be a (possibly unbalanced) 
multiarrangement in $V=\K^2$ with $\exp\Am=(d_1,d_2)$ and $d_1 \le d_2$. Then for any $H\in\A$  which is not heavy, it holds that 
\begin{equation}\label{Ineq3}
m(H) \le d_1 \le d_2 \le |m|-m(H).
\end{equation}
In particular, 
\begin{equation}\label{Ineq1}
b_2\Am = d_1 d_2 \ge m(H)(|m| - m(H))
\end{equation}
for any multiarrangement $\Am$ of rank two and $H\in H$.
\end{lemma}

\begin{proof}
Fix a hyperplane $H\in\A$ which is not heavy in $\Am$, i.e., $m(H)\le |m|-m(H)$. Set $k:=|m|-2m(H)$.
Then $(\A,m+k \delta_H)$ is unbalanced with $\exp (\A,m+k \delta_H) = (m(H)+k,m(H)+k)$.
Note that~\eqref{Ineq3} holds for $(\A,m+k \delta_H)$. Now 
taking $k \to 0$ with Lemma \ref{LemIneq3} completes the 
proof. 
\end{proof}

Now we prove Theorems \ref{ineq2}, \ref{heavyflag} and Corollary \ref{combin}. 

\begin{proof}[Proof of Theorem~\ref{ineq2}]
If $H$ is heavy in $\Am$ we have $ B_2^H\Am = b_2^H \Am$ and $ b_2(\A^H,m^H ) = b_2(\A^H,m^*)$. Hence the inequality was already shown in Theorem~\ref{AYheavy}. Now assume that $H$ is not heavy in $\Am$. We proceed in the following three steps.
\begin{itemize}
\item[$(i)$] $ b_2^H\Am \ge B_2^H\Am$: By the 
inequality~\eqref{Ineq1} of Lemma~\ref{LemIneq2}, we obtain
\begin{align*}
&b_2^H\Am - B_2^H\Am  \\
=& \sum_{X\in L_2(\A)} b_2(\A_X,m_X) -m(H)(|m|-m(H))  -\sum_{X\in L_2(\A), X \not \subset H} b_2(\A_X,m_X)\\
=& \left(\sum_{X\in L_2(\A), X \subset H} b_2(\A_X,m_X)\right)-m(H)(|m|-m(H))\\
\stackrel{\eqref{Ineq1} }{\ge} & \left(\sum_{X\in L_2(\A), X \subset H} m(H)(|m_X|-m(H))\right)- m(H)(|m|-m(H))\\
=& m(H) \left( \sum_{X\in L_2(\A), X \subset H} (|m_X|-m(H)) - (|m|-m(H)) \right)  = 0.
\end{align*}
\item[$(ii)$] $ B_2^H\Am \ge b_2(\A^H,m^H)$: The quantities 
$ B_2^H\Am$ and $ b_2(\A^H,m^H)$ are clearly independent of $m(H)$. Hence 
we may assume 
that $H$ is the heavy hyperplane in $\Am$. 
Then $ B_2^H\Am \ge b_2^H \Am$ has already been proven in Theorem \ref{AYheavy}.
\item[$(iii)$] $ b_2(\A^H,m^H ) \ge b_2(\A^H,m^*)$: By the local-global formula in Proposition~\ref{PrePro2} we can reduce to the case $\A^H$ is an arrangement of rank two. 
The inequality~\eqref{Ineq3} and the definition of the Euler multiplicity 
show that $m^H(X)\ge m^*(X)$ for any $X \in \A^H$. Then
Lemma~\ref{LemIneq3} 
completes the proof.
\end{itemize}
\end{proof}

\begin{proof}[Proof of Theorem~\ref{heavyflag}]
Agree that $\A=\left(\A^{X_0},m^{X_0}\right)$, where $m^{X_0} (H) =1$ for all $H\in\A$. Theorem~\ref{PreTh5} applied to $\A$ yields
\begin{equation}\label{Eq18}
\A \mbox{ is free} \Leftrightarrow \left( \A^{X_1},m^{X_1} \right) \mbox{ is free and } b_2(\A) -\left( \left| m^{X_0} \right|-1 \right) = b_2 \left(\A^{X_1},m^{X_1}\right).
\end{equation}
Note that $\left(\left(\A^{X_i}\right)^{X_{i+1}},\left(m^{X_i}\right)^{X_{i+1}}\right) = \left(\A^{X_{i+1}},m^{X_{i+1}}\right)$ and the multiarrangements $\left(\A^{X_{i+1}},m^{X_{i+1}}\right)$ are unbalanced with the heavy hyperplane $X_{i+2} \in \A^{X_{i+1}}$ for $i=1,\ldots,\ell-3$. 
Hence we can apply Theorem~\ref{AYheavy} to obtain, 
for $i=1,\ldots,\ell-3$,
\begin{align}
\label{Eq19} &\left(\A^{X_{i}},m^{X_{i}}\right) \mbox{ is free} \Leftrightarrow \\
\nonumber &\left(\A^{X_{i+1}},m^{X_{i+1}} \right)\mbox{ is free and } b_2^{X_{i+1}} \left( \A^{X_{i}},m^{X_{i}} \right) = b_2 \left(\A^{X_{i+1}},m^{X_{i+1}}\right).
\end{align}
Since $\left(\A^{X_{\ell -2}},m^{X_{\ell -2}}\right)$ is a multiarrangement of rank 2, it is free by Proposition~\ref{PrePro1}. So we can iteratively link the statements~\eqref{Eq18} and~\eqref{Eq19} to obtain
\begin{equation}\label{Eq20}
\A \mbox{ is free} \Leftrightarrow b_2^{X_{i+1}} \left( \A^{X_{i}},m^{X_{i}} \right) = b_2 \left(\A^{X_{i+1}},m^{X_{i+1}}\right) \mbox{for all } i=0,\ldots ,\ell -3.
\end{equation}
Since Theorem~\ref{PreTh5} and Theorem~\ref{AYheavy} imply that $b_2^{X_{i+1}} \left( \A^{X_{i}},m^{X_{i}} \right) \ge b_2 \left(\A^{X_{i+1}},m^{X_{i+1}}\right)$ for all $i=0,\ldots ,\ell -3$, the right-hand side of~\eqref{Eq20} is equivalent to
\begin{equation}\label{eq100}
\sum_{i=0}^{\ell -3} b_2^{X_{i+1}} \left( \A^{X_{i}},m^{X_{i}} \right) = \sum_{i=0}^{\ell -3} b_2 \left(\A^{X_{i+1}},m^{X_{i+1}}\right).
\end{equation}
Noting that 
\begin{eqnarray*}
b_2^{X_{i+1}}(\A^{X_{i}},m^{X_{i}})&=&
b_2(\A^{X_i},m^{X_i})-m^{X_i}(X_{i+1})(|m^{X_i}|-m^{X_i}(X_{i+1}))
\end{eqnarray*}
we find that~\eqref{eq100} is in fact equivalent to
\begin{equation}\label{Eq21}
 b_2 (\A) - \sum_{i=0}^{\ell -3} m^{X_{i}}(X_{i+1}) \left( \left|m^{X_{i}}\right|-m^{X_{i}}(X_{i+1} )\right) =b_2 \left(\A^{X_{\ell -2}},m^{X_{\ell-2}}\right) .
\end{equation}
Since $\exp \left(\A^{X_{\ell -2}},m^{X_{\ell -2}}\right)=\left( m^{X_{\ell -2}}(X_{\ell-1}), \left|m^{X_{\ell-2}}\right|-m^{X_{\ell-2}}(X_{\ell-1} ) \right)$ we have 
\[b_2 \left(\A^{X_{\ell -2}},m^{X_{\ell -2}}\right)= m^{X_{\ell -2}}(X_{\ell-1}) \left( \left|m^{X_{\ell-2}}\right|-m^{X_{\ell-2}}(X_{\ell-1} )\right).\]
This shows that~\eqref{Eq21} is equivalent to
\[
b_2 (\A) = \sum_{i=0}^{\ell -2} m^{X_{i}}(X_{i+1}) \left( \left|m^{X_{i}}\right|-m^{X_{i}}(X_{i+1} )\right)  .
\]
Since $\left|m^{X_{i}}\right|=\sum_{j=i}^{\ell-1} m^{X_{j}} (X_{j+1})$, the above 
equality is equivalent to
\begin{align*}
b_2 (\A) &= \sum_{i=0}^{\ell -2} m^{X_{i}}(X_{i+1}) \left( \sum_{j=i+1}^{\ell-1} m^{X_{j}} (X_{j+1})  )\right) \\
& = \sum_{0 \le i< j \le \ell-1}m^{X_i}(X_{i+1})m^{X_j}(X_{j+1})
\end{align*}
This completes the proof, since by~\eqref{Eq20} this is equivalent to $\Am$ being free.
\end{proof}

\begin{proof}[Proof of Corollary~\ref{combin}]
Corollary~\ref{combin} follows directly from Theorem~\ref{heavyflag}, since the characteristic polynomial $\chi (\A;t)$, and in particular also $b_2(\A)$, of a simple arrangment $\A$ is combinatorially determined. The same holds true for the multiplicities of the Euler-Ziegler restrictions. 
For the supersolvablity, use Proposition 4.2 in \cite{A3}.
\end{proof}

Corollary~\ref{combin} says that 
Terao's conjecture holds true in the category of arrangements with a heavy flag. 

Furthermore, we can show the following theorem, which links the freeness of an unbalanced multiarrangement in $\K^3$ with its combinatorics.

\begin{cor}\label{Cor3Multi}
Let $\Am$ be a multiarrangement in $V=\K^3$ with a heavy hyperplane $H_0\in \A$. Assume that 
$\left(\A^{H_0},m^{H_0}\right)$ and $(\A_X,m_X)$ are either unbalanced, or consist of up to three hyperplanes for all $X\in L_2(\A)$. Then the freeness of $\A$ depends only on the combinatorics of $\Am$, i.e., on $L(\A)$ and the multiplicity function $m$.
\end{cor}
\begin{proof}
Let $\Bm$ be a multiarrangement of rank two. If $\Bm$ has the 
heavy hyperplane $H\in \B$, 
then $\exp \Bm = (m(H),|m|-m(H))$. 
If $\B = \{H_1,H_2\}$, then $\exp \Bm = (m(H_1),m(H_2))$.
If $|\B| = 3$ and $\Bm$ is balanced, Wakamiko showed in~\cite{W}, that $\exp \Bm = (\lfloor \frac{|m|}{2}\rfloor,\lceil\frac{|m|}{2}\rceil)$.
Hence in all cases, the exponents are determined combinatorially.
So by the assumption, 
$b_2 \left(\A^{H_0},m^{H_0}\right)$ and $b_2 (\A_X,m_X)$ for all $X\in L_2(\A)$ are determined combinatorially as well. By the local-global formula in Proposition~\ref{PrePro1}, 
so is $b_2^{H_0}\Am$. Since $\left(\A^{H_0},m^{H_0}\right)$ in $\K^2$ is free, we can determine 
the freeness of $\A$ combinatorially with Theorem~\ref{AYheavy}.
\end{proof}

For the case of the Weyl arrangement of the type $A_3$, all non-trivial localizations and restrictions have at most three hyperplanes. So Corollary~\ref{Cor3Multi} immediately implies Corollary~\ref{A3} from the introduction. 

Next, we consider the following example of a Weyl multiarrangement of type $A_3$ to show that the heaviness assumptions in Theorem~\ref{AYheavy} can not be dropped in general.

\begin{example}
Let $(\A,m)$ be defined by 
$$
x^3y^3z^3(x-y)^3
(y-z)^3(z-x)^3=0.
$$
Then this is free (see \cite{T} for example), but 
$b_2^H(\A,m)=107-3(18-3)=62
>B_2^H(\A,m)=58>
b_2(\A^H,m^H)=56$ for any $H\in\A$. 

\end{example}

\section{Multi-freeness and supersolvability}

In this section we study the relation between the freeness of multiarrangements $(\A,m)$ and 
the geometry of the underlying 
simple arrangement $\A$. For example, if any multiplicity on $\A$ is free, then 
$\A$ is known to be the product of one and two-dimensional simple arrangements (\cite{ATY}). Hence 
freeness affects the geometry of the simple arrangement on which multiplicities are. Here, by using 
the results of the previous sections, we show a similar result in a wider category of free arrangements. 

\begin{theorem}\label{heavycor}
Let $(\A,m)$ be an 
essential 
free multiarrangement in $V=\K^\ell$ with $X_i \in L_i(\A)$ for $i=0,\ldots,\ell$ such that $X_i$ is the heavy hyperplane in $(\A^{X_{i-1}},m^{X_{i-1}})$ for $i=1,\ldots,\ell-1$. 
Then $\A$ is supersolvable (hence free) with a supersolvable filtration $\A_{X_1} \subset \cdots \subset \A_{X_{\ell-1}}\subset \A_{X_{\ell}}=\A$.
\end{theorem}
\begin{proof}
Set $m_i:=m^{X_{i-1}}(X_i)$ and define the filtration $\A_{X_1} \subset \cdots \subset \A_{X_{\ell-1}}\subset \A_{X_{\ell}}=\A$. 
Now let $\A_{X_i}\setminus \A_{X_{i - 1}} =\{H^i_1,\ldots ,H^i_{t_i}\}$ for $i=1,\ldots,\ell$ and some $t_i$.
First, by Theorem~\ref{AYheavy}, 
$(\A^{X_1},m^{X_1})$ is free with $\exp (\A^{X_1},m^{X_1}) =(m_2,\ldots,m_{\ell})$ and 
\begin{equation}\label{heavycoreq}
b_2^H (\A,m) = b_2 (\A^{X_1},m^{X_1}) = \sum_{2\le i<j\le \ell} m_i m_j.
\end{equation}
Second, we can show that 
\[
\sum_{j=1}^{t_i} m(H^i_j) = m_i.
\]
To show the above, note that 
$$
\A_{X_i} \setminus \A_{X_{i-1}}=\{H \in \A \mid 
X_{i-1} \not \subset H,\ H \cap X_{i-1}=X_i\}.
$$
Hence the definition of the Ziegler restriction ensures that 
$$
\sum_{j=1}^{t_i} m(H^i_j) = m^{X_{i-1}}(X_i)=m_i.
$$
Set $L^{away}:=\{ Y\in L_2(\A)\mid X_1 \not \supset Y\}$ and $L^{diff}:=\{ Y\in L^{away} \mid  H^i_s \cap H^j_k=Y \mbox{ for some } i\neq j \mbox{ and any } s,k\}$.
By Proposition \ref{PrePro2}, 
\begin{equation}\label{Laway}
b_2^H\Am =\sum_{Y \in L^{away} } b_2(\A_Y,m_Y)\ge \sum_{Y \in L^{diff} } 
b_2(\A_Y,m_Y).
\end{equation}
Define $L_i^{diff} := \{ Y\in L^{diff} \mid   H^i_s \cap H^j_k=Y \mbox{ for some } i < j \mbox{ and any } s,k\}$ for $i=2,\ldots,\ell-1 $.
Then by~\eqref{Ineq1}
\begin{equation}\label{Ldiff}
\sum_{Y \in L^{diff} } b_2(\A_Y,m_Y) = \sum_{i=2}^{\ell - 1} \sum_{Y \in L_i^{diff} } b_2(\A_Y,m_Y) \ge \sum_{i=2}^{\ell - 1} m_i \sum_{i<j}^\ell  m_j = \sum_{2\le i<j\le \ell} m_i m_j.
\end{equation} 
By~\eqref{heavycoreq},~\eqref{Laway} and~\eqref{Ldiff} are both equalities. So
\[
\sum_{Y \in L^{away}\setminus L^{diff}} b_2(\A_Y,m_Y)=0,
\]
i.e., $L^{away}=L^{diff}$.
Hence if $H^i_s \cap H^i_r =Y$ for some $s\neq r$ and $Y \in L^{away}$, then 
there is $H^j_k\supset Y$ with $i\neq j$. 
Assume $j>i$. Then $\codim H^i_s\cap H^i_t\cap H^j_k=3$ by the definition of the filtration.
Hence $Y\in L_3(\A)$, which is a contradiciton. Thus $j<i$ in this situation, which shows that the given filtation is a supersolvable filtration.
\end{proof}


As a corollary, we can relate the non-supersolvability with non-freeness and heavy flags.

\begin{cor}\label{cor1}
Let $\Am$ be a multiarrangment with a heavy flag, as in the above notation. If $\A$ is not supersolvable, then $\Am$ is not free.
\end{cor}

The Euler restriction of a simple free arrangement is not necessarily free 
shown by Edelman and Reiner in \cite{ER}. 
However, in our setting we can show the following.

\begin{cor}\label{cor2}
Let $\A$ be a free arrangement and $\{X_i\}_{i=0}^\ell$ be a heavy flag. 
Then $\A^{X_i}$ is supersolvable for all $i=1,\ldots,\ell-1$. In particular, $\A^{X_i}$ is 
free for all $i=1,\ldots,\ell-1$.
\end{cor}

\begin{proof}
By Theorem~\ref{AYheavy} the unbalanced multiarrangments $(\A^{X_i},m^{X_i})$ are free and satisfy the conditions of Theorem~\ref{heavycor}. Hence Theorem~\ref{heavycor} completes the 
proof.
\end{proof}
\begin{example}
Let $(\A,m)$ be defined by 
$$
x^1y^2z^7(x+y)^1
(y-z)^2(z-x)^1=0.
$$
Then $\Am$ has the heavy flag $H:=\{z=0\} \in \Am$ and $\{y=z=0\} \in (\A^H,m^H)$ but it can easily be checked that the arrangement $\A$ is not supersolvable (assume $\mbox{char } \K\neq 2$). Hence, $\Am$ can not be free by Corolary~\ref{cor1}.

Let $\B$ in $V=\K^4$ be defined by
$$
xyw(x-z)(y-z)(y-w)\prod_{k=-2}^2(z-kw)=0.
$$
Then $\B$ has the heavy flag $H:=\{w=0\} \in \B$, $L:=\{w=z=0\} \in (\B^H,m^H)$ and $K:=\{w=z=y=0\} \in (\B^L,m^L)$, with 
$m^V(H)=1$, $m^H(L)=5$, $m^L(K)=3$ and $m^K(0)=2$. So the RHS of~\eqref{flagEq} evaluates to $1(5+3+2)+5(3+2)+3\cdot 2=41$.
Further, we can compute $b_2(\B)=41$, which implies that $\B$ is free with  $\exp(\A)=(1,5,3,2)$ by Theorem~\ref{heavyflag}.

Hence, Corolary~\ref{cor2} implies that the simple arrangements $\B^H$ and $\B^L$ are supersolvable (hence free) as well.
\end{example}

\end{document}